\documentclass[]{article}

\usepackage{makeidx}  % allows for indexgeneration

\usepackage{amsthm} % For the `proof' environment

\usepackage{latexsym}
\usepackage[pdftex]{graphics}
\usepackage{xypic}
\usepackage{amsmath}
\usepackage{amssymb}
\xyoption{all} % for odd diagrams ...
\usepackage{mathrsfs}
\usepackage{eufrak}

\newtheorem{theorem}{Theorem}[section]

\newtheorem{definition}[theorem]{Definition}

\newtheorem{proposition}[theorem]{Proposition}

\newtheorem{remark}[theorem]{Remark}

\newcommand{\C}{{\mathcal C}}
\newcommand{\J}{{\mathcal J}}

\begin{document}

\author{Peter M. Hines}
%\address{Peter Hines \\ Dept. of Computer Science, \\ University of York, \\ York, U.K. \\ YO10 5DD } 
 
 %\institute{University of York}
 
 \title{Identities in modular arithmetic from \\ reversible coherence operations}

\maketitle
\begin{abstract}
This paper investigates some issues arising in categorical models of reversible  logic and computation. Our claim is that the structural (coherence) isomorphisms of these categorical models, although generally overlooked, have decidedly non-trivial computational content. The theory of categorical coherence is based around reversible structural operations (canonical isomorphisms) that allow for transformations between related, but distinct, mathematical structures.  A number of {\em coherence theorems} are commonly used to treat these transformations as though they are identity maps, from which point onwards they play no part in computational models. We simply wish to point out that doing so overlooks some significant computational content.%and (for example) treat associativity up to isomorphism as though it is a strict identity. However, doing so changes % they are strict in reasonable circumstances 

%This paper demonstrates that these reversible operations, and the related coherence theorems have non-trivial computational content.  
We give a single example (taken from an uncountably infinite set of similar examples, and based on structures used in models of reversible logic and computation) of a category whose structural isomorphisms manipulate modulo classes of natural numbers. We demonstrate that the coherence properties that usually allow us to ignore these structural isomorphisms in fact correspond to countably infinite families of non-trivial identities in modular arithmetic.  Further, proving the correctness of these equalities without recourse to the theory of categorical coherence appears to be a hard task. 
\end{abstract}

\section{Introduction}\label{intro} 
\subsection{Historical background}
In \cite{LL}, J.-Y. Girard introduced {\em Linear Logic}, a striking new decomposition of classical logic. By contrast to previous approaches to logic, it was based around the twin related principles of {\em reversibility} and {\em resource-sensitivity}. Although the structural operations of copying and contraction (i.e. deletion against a copy) were not completely abandoned (as in sub-structural logics \cite{FP}), they were severely restricted. Via the Curry-Howard isomorphism \cite{MSPU} (also known as the `proofs-as-programs' correspondence) linear logic also has a close connection with reversible and resource-sensitive versions of computing systems such as lambda calculus and combinatory logic \cite{AHS}.

The computational interpretation was pushed further in the {\em Geometry of Interaction} program \cite{GOI0}, giving related models of linear logic \cite{GOI1,GOI2} (see also \cite{DR}). Although these models were degenerate in the logical sense (they identified conjunction with disjunction, and existential quantification with universal quantification) their computational content remained, as demonstrated by a series of practical computational interpretations in \cite{GOI2}. (As shown later \cite{PH03}, the dynamical part of the Geometry of Interaction system was implemented using precisely the same tools required to model reversible (space-bounded) Turing machines).

A significant challenge for logicians at this point was to give categorical models of both Linear Logic and the (related but distinct) Geometry of Interaction system, following the close correspondence between logics / type systems, and closed categories pioneered by \cite{LS}. For the purposes of this paper, we concentrate on the more computationally oriented Geometry of Interaction.

Several authors \cite{SA96,PH97,JSV} noted that the dynamical, or computational, part of the Geometry of Interaction system was  a form of {\em compact closure} \cite{KL} arising from categorical constructions \cite{JSV,SA96} on the category of partial reversible functions. As pointed out in \cite{PH97,PH99,AHS} (and implicit in \cite{GOI1}), the Geometry of Interaction is an essentially untyped (in the sense of $\lambda$-calculus) reversible computational system --- this is a consequence of the requirements of reversibility and resource-sensitivity. Any categorical interpretation must take this into account.

\subsection{The purpose of this paper}
The purpose of this paper is simply to point out some previously overlooked, decidedly non-trivial, computational content that arises in these models (in fact, familiarity with the logical models and computational systems listed above is not a requirement for understanding this paper -- but does help place the theory firmly within its historical context). The computational content comes, not from the dynamics of the GoI system (i.e. compact closure in categories of partial reversible functions), but simply from the fact that the system in question is untyped. Categorically, a model of an untyped system is a category with precisely one object (i.e. a monoid). Thus the GoI system is modelled within a monoid of partial reversible functions. 

In categorical logic / categorical models of computation, it is standard to ignore completely a class of structural isomorphisms known as {\em coherence isomorphisms}. There is a formal justification for doing so --- any category with non-trivial structural isomorphisms is equivalent (in a very precise sense) to one with trivial structural isomorphisms \cite{MCL}. 

However, there is a subtlety that is often overlooked; the process of constructing this equivalent category with trivial structural isomorphisms involves modifying the collection of objects of the category (\& hence, by the correspondence between categories and logics pioneered in \cite{LS}, modifying the type system). An appendix to \cite{PH13} (see also \cite{PHarxiv}) makes clear what this means for untyped systems; the `equivalent' version with trivial structural isomorphisms has a countably infinite class of objects, and thus is no longer type-free. 

As type-freeness is such an essential component of the GoI system, we are thus forced to deal with these structural isomorphisms -- this paper studies a set of such isomorphisms that arise implicitly in \cite{GOI1}. We demonstrate that, although  the category itself has only one object, modulo classes of integers play the same r\^ole as objects in this untyped setting. Thus, the structural isomorphisms correspond to (highly non-trivial) identities in modular arithmetic. Further, the classic theory of coherence that usually allows us to ignore structural isomorphisms completely in this case allows us to derive infinite sets of identities in modular arithmetic, essentially for free.

\subsection{Categorical identities up to isomorphism}
In category theory, especially the theory of monoidal categories, {\em coherence isomorphisms} are reversible structural operations that transform objects of categories (frequently, concrete mathematical structures) into isomorphic objects that differ only by a simple structural equivalence. 

The canonical example, of course, is {\em associativity}, where for foundational reasons one must replace the strict identity $X\otimes (Y\otimes Z) = (X\otimes Y)\otimes Z$ by a pair of mututally inverse isomorphisms \[ \xymatrix{ X\otimes (Y\otimes Z) \ar@/^9pt/[r]^{\tau_{X,Y,Z}} &  (X\otimes Y)\otimes Z \ar@/^9pt/[l]^{\tau_{X,Y,Z}^{-1}}} \] 
These natural isomorphisms are required to satisfy a family of {\em coherence conditions} that ensure that any such re-bracketing is both reversible and confluent. 

The distinction between a strict structural property (based on equality) and one that holds up to isomorphism is subtle, and a variety of coherence theorems \cite{MCL} tell us that for all practical purposes, we may ignore this subtlety, and treat properties such as associativity as though they are strict. 
However, a passing comment in the appendix of \cite{PH13} (expanded upon in a talk given by the author at Dagstuhl Seminar 12352, `Information Flow and its Applications' \cite{DAG}) observes that in various settings, these structural isomorphisms are concrete reversible arithmetic operations and the very coherence theorems used to ignore them have non-trivial computational content.  

This paper expands upon these observations via a simple representative example. We give an untyped (i.e.  single-object) unitless monoidal category whose structural isomorphisms are based on modular arithmetic, and then describe the significant computational advantage that the theory of categorical coherence provides in decisions procedures for equality of such reversible operations. In particular, we demonstrate that categorical diagrams based on $k$ distinct nodes correspond to arithmetic identities over equivalence classes of the form $\{ 2^k.\mathbb N + x  \}_{x=0\ldots 2^k-1}$.  Despite this, the coherence theorem for associativity  provides, for free, a large (countably infinite) class of arithmetic identities over such modular clases that are guaranteed to be correct. At least to the author, these identities are not readily apparent simply from their algebraic description. % ,however categorical 

\subsection{MacLane's coherence theorem for associativity, and untyped monoidal categories}
MacLane's coherence theorem for associativity is commonly, although incorrectly, described as stating that `all diagrams built from coherence isomorphisms commute'.  This is a correct characterisation of the more technical result in some, but certainly not all,  cases (in particular simple calculations will demonstrate that it does not hold the for constructions of this paper). The distinction becomes important when the objects of the category do not satisfy a `freeness' condition with respect to the monoidal tensor, leading to what \cite{MCL} refers to as {\em undesirable identifications between objects}. %\footnote{Although, of course, what is undesirable in certain contexts may be absolutely essential to, for example, researchers into untyped logical or computing systems.}. 
Thus, when the class of objects is not only a set, but is {\em finite}, the informal characterisation above can never coincide with the formal statement of the theorem. 

This paper presents a rather extreme example of this:  we exhibit a small (unitless) symmetric monoidal category with exactly one object $N$ satisfying the equality\footnote{Note that this is strict equality, rather than isomorphism. For category-theorists worried about foundational questions related to a notion of equality between objects, we emphasise that this is a {\em small} category. Although $N\otimes (N\otimes N) = (N\otimes N)\otimes N$, this equality of objects does not imply that the corresponding associativity isomorphism is a strict identity.} $N\otimes N = N$, as in the example of J. Isbell  used by MacLane to motivate the notion of coherence up to isomorphism \cite{MCL} p. 160.

This `untyped' monoidal category is an example of a general construction introduced in \cite{PH97,PH99} -- see also \cite{PH13,PHarxiv}. As demonstrated in an Appendix to \cite{PH13}, there are uncountably many such untyped monoidal categories based on functions on $\mathbb N$ (in 1:1 correspondence with the interior points of the Cantor set, excluding a subset of measure zero), of which the one we present is merely the simplest.

We then describe which canonical diagrams of this category are predicted to commute by MacLane's coherence theorem for associativity, and demonstrate that these are non-obvious identities in modular arithmetic.  

%This raises the natural question of those canonical diagrams for which no prediction is made by MacLane's coherence theorem for associativity  --- do any, all, or none of these remaining diagrams commute? Part of the purpose of this paper is to provide motivation for a subsequent series of papers that answers this question, via coherence theorems and (constructive) decision procedures for such diagrams in a more general setting.
 
\section{An untyped monoidal category}\label{sums}
We first give some simple arithmetic constructions on $\mathbb N$, based on arithmetic modulo $2^k$, for $k\in \mathbb N$, with a close connection to the theory of symmetric monoidal categories \cite{MCL}:

\begin{definition}\label{prelim}
Let us denote the monoid of bijections on the natural numbers by $\J$, and treat this as a single-object category. 
We define $\tau,\sigma\in \J = \J (\mathbb N,\mathbb N)$  as follows:
\[ \tau (n) =\left\{ \begin{array}{lr}
2n &  \ \ \ n\ (mod \ 2)=0, \\
%		& 	\\
n+1 &  \ \ \ n\ (mod \ 4)=1,  \\	
%		& 	\\
\frac{n-1}{2} &  \ \ \  n\ (mod \ 4)=3.  \\	
\end{array}\right.
\] 
\[ 
\sigma (n)  \ = \ \left\{\begin{array}{lr} 
n+1 & n \mbox{ even,} \\
n-1  & n \mbox{ odd.} 
\end{array} \right. 
\]
We also give an operation that, given two bijections on $\mathbb N$, returns another bijection. Given arbitrary 
$f,g\in \J$, % be bijections. Another bijection, $f\star g:\mathbb N\rightarrow \mathbb N$ may be defined by
we define 
\[ (f\star  g) (n) \ = \ \left\{\begin{array}{lr} 
2.f\left(\frac{n}{2}\right) & n \mbox{ even,} \\
 
2.g\left(\frac{n-1}{2}\right) + 1  & n \mbox{ odd.} 
\end{array} \right. 
\]
\end{definition}
The following properties of the above bijections and operations will be established via basic modular arithmetic. These properties are, as will be apparent, closely related to the structural properties and coherence conditions of symmetric monoidal categories:
\begin{proposition}\label{coherent}
Let $(\_ \star \_ ), \sigma ,\tau$ be as in Definition \ref{prelim} above. Then for all $f,g,h\in \J$, the following properties hold:
\begin{enumerate}
\item\label{1} {\bf Identities} $id  \star id  = id $
\item\label{2} {\bf Interchange} $(h\star k)(f\star g) = (hf\star kg)$
\item\label{3} {\bf Natural associativity} $\tau(f\star (g\star h))=((f\star g)\star h)\tau$
\item\label{4} {\bf Natural symmetry} $\sigma(g\star f) = (f\star g)\sigma$
\item\label{5} {\bf Pentagon} $\tau^2 = (\tau\star id )\tau (id \star \tau)$
\item\label{6} {\bf Hexagon} $\tau\sigma \tau = (\sigma \star id ) \tau (id  \star \sigma)$
\end{enumerate}
\end{proposition}
\begin{proof}$ $ \\
\begin{enumerate}
\item 
By definition, $(id \star id)(n) = \left\{\begin{array}{lr} 2\left( \frac{n}{2} \right)\ = \ n & n \mbox{ even, } \\  & \\ 2\left( \frac{n-1}{2} \right) + 1\ =\ n & n \mbox{ odd. } \end{array}\right.$
%Thus $(1\star 1)(n) = n$.
\item Similarly, 
$ (h\star k)(f\star g) (n) = \left\{ \begin{array}{lr} (h\star k)\left( 2f\left(\frac{n}{2}\right)\right)  & n \mbox{ even, } \\  & \\ (h\star k) \left( 2g\left( \frac{n-1}{2}\right) +1  \right) & n \mbox{ odd. } \end{array}\right. $

Now observe that $2f\left( \frac{n}{2} \right)$ is always even, for arbitrary choice of $f\in {\bf Bij}(\mathbb N,\mathbb N)$ and even $n\in \mathbb N$. Similarly, $2g\left( \frac{n-1}{2}\right) +1$ is always odd, for arbitrary choice of $g\in {\bf Bij}(\mathbb N,\mathbb N)$ and odd $n\in \mathbb N$. Thus
\[ (h\star k)\left( (f\star g) (n)\right)  = \left\{ \begin{array}{lr}  2h\left( \frac{2f\left(\frac{n}{2} \right)}{2}\right) & n \mbox{ even, } \\   & \\
2k\left( 
	\frac{
		\left( 2g \left( \frac{n-1}{2} \right) +1 \right) -1}{2} \right) +1 & n \mbox{ odd.} \end{array} \right. \]
Simplifying this expression,
\[ (h\star k)(f\star g)(n) = (hf\star kg)(n)= \left\{ \begin{array}{lr} 2hf\left( \frac{n}{2} \right) & n \mbox{ even, } \\   2kg \left( \frac{n-1}{2} \right)+1 & n \mbox{ odd. }\end{array}\right. \]
%and so $(h\star k)(f\star g) = (hf \star kg)$, as required.
\item We first establish explicit formul\ae\ for $f\star(g \star h)$ and $(f\star g)\star h$. 
By definition, 
\[ (f\star (g\star h))(n) =\left\{ \begin{array}{lr} 2f\left(\frac{n}{2}\right) & n \mbox{ even, } \\ (g\star h)\left(\frac{n-1}{2}\right) +1 & n \mbox{ odd. } \end{array}\right. \]
Unwinding the definition of $(g\star h)$,
\[ (g\star h)\left( \frac{n-1}{2} \right) \ = \ \left\{ \begin{array}{lr} 2g\left(\frac{n-1}{4}\right) & \frac{n-1}{2} \mbox{ even, } \\   & \\
										2h\left( 
												\frac{
													\left(\frac{n-1}{2}\right) -1 }
													{2}
												\right)
											+1 
										& \frac{n-1}{2} \mbox{ odd. }
							\end{array}\right.
\]
Thus 
\[ (f \star (g\star h))(n) = \left\{ 
\begin{array}{lr}
			2f\left( \frac{n}{2} \right) & \ \ \ \ n\ (mod \ 2) \ =\  0, \\	 
			2g\left( \frac{n-1}{4} \right) +1 & \ \ \ \  n\ (mod \ 4) \ = \  1, \\	 
			2h\left( \frac{n-3}{4} \right) + 3 & \ \ \ \  n \ (mod \ 4) \ = \ 3. 
	\end{array} \right. \]
Using similar reasoning,
\[ ((f\star g)\star h)(n) = \left\{ \begin{array}{lr} 
				4f\left(\frac{n}{4}\right) & \ \ \ \  n\ (mod \ 4) \ = \ 0 \\  
				4g\left(\frac{n-2}{4}\right)+2 & \ \ \ \  n\ (mod \ 4) \ = \ 2 \\  
				2h\left( \frac{n-1}{2}\right) +1 & \ \ \ \  n \ (mod \ 2) \ = \ 1
				\end{array}
			\right.
\]
From the explicit description of $\tau$, 
\[ 
\tau (f\star (g\star h)) = \left\{ \begin{array}{lr}  
				4f\left( \frac{n}{2}\right)	& \ \ \ \  n\ (mod\ 2) \ = \ 0 \\	 
				4g\left( \frac{n-1}{4}\right)+2	& \ \ \ \  n\ (mod\ 4) \ = \ 1 \\	 
				2h\left( \frac{n-3}{4}\right) +1& \ \ \ \  n\ (mod\ 4) \ = \ 3
				\end{array}\right.
			\]
and an almost identical calculation will verify that $((f\star g)\star h)\tau$ is given by the same formula.
\item Direct calculation gives that 
\[ \sigma(g\star f)(n)=(f\star g)\sigma(n) = \left\{ \begin{array}{lr}
				2f\left(\frac{n}{2}\right) + 1 & \ \ \ \  n\ (mod\ 2) = 0 \\ 
				2g\left(\frac{n-1}{2}\right) 	& \ \ \ \  n\ (mod\ 2) = 1 \end{array}\right. \]
%which may be seen to be equal to $\sigma(g\star f)(n)$
\item We first describe the individual parts of the Pentagon equation:
\[ (id \star \tau)(n) =\left\{ \begin{array}{lr}
			n 	& \ \ \ \  n\ (mod\ 2) \ = \ 0 \\  		
			2n-1	& \ \ \ \  n\ (mod\ 4) \ = \ 1 \\  	
			n+2 	& \ \ \ \  n\ (mod\ 8) \ = \ 3 \\  	
			\frac{n-1}{2} 	& \ \ \ \  n\ (mod\ 8) \ = \ 7 
			\end{array}\right.
\] 
Similarly,
 \[ 
  (\tau \star id )(n) =\left\{ \begin{array}{lr}
		2n 			&  \ \ \ \ n\ (mod \ 4) \ = \ 0 \\ 	 
		n+2 			& \ \ \ \  n\ (mod \ 8) \ = \ 2 \\ 	 
		\frac{n+1}{2} 	&  \ \ \ \ n\ (mod \ 8) \ = \ 6 \\  
		n 			&  \ \ \ \ n\ (mod \ 2) \ = \ 1 \\ 
\end{array}\right.
 \]
Composing, on a case-by-case basis, gives 
\[ \tau^2(n)=(\tau\star id )\tau (id  \star \tau) (n) = 
\left\{ \begin{array}{lr} 
				4n &  \ \ \ \ n\ (mod \ 2) \ = \ 0 \\  
				n+2 &  \ \ \ \ n\ (mod \ 4) \ = \ 1 \\  
				\frac{n+1}{2} &  \ \ \ \ n\ (mod \ 8) \ = \ 3 \\  
				\frac{n-3}{4} &  \ \ \ \ n\ (mod \ 8) \ = \ 7 
	\end{array}
	\right.
	\]
\item For the hexagon equation, direct calculations (that by this stage, we are happy to leave as an exercise) demonstrate that 
\[ \tau\sigma\tau(n) = (\sigma \star id ) \tau (id  \star \sigma)(n) = 
\left\{ \begin{array}{lr}
			2n+2 &  \ \ \ \ n\ (mod\ 2) \ = \ 0 \\  
			\frac{n+1}{2} &  \ \ \ \ n\ (mod\ 4) \ = \ 1 \\  
			n-3 &  \ \ \ \ n\ (mod\ 4) \ = \ 3
\end{array}\right.
\]
\end{enumerate}
\end{proof}

\begin{remark}
$\mathcal J$ is a monoid --- a one-object, or single-typed, category.  Despite this, the above calculations demonstrate how the r\^ole of distinct objects in the theory of symmetric monoidal categories is instead played by certain subsets of $\mathbb N$ --- the congruence classes of the form $\{ 2^k.\mathbb N + x  \}_{x=0\ldots 2^k-1}$.  % inverse semigroup theory.
%There are clearly some connections between number theory and the theory of coherence within monoidal categories that remain to be explored. This is the subject of ongoing work. 
\end{remark}

As demonstrated in Proposition \ref{coherent} above, $(\mathcal J, \star ,\tau ,\sigma)$ has all the structure of a symmetric monoidal category, except for the existence of a unit object. We axiomatise such situations as follows:

\begin{definition}\label{semicat}
Let $\C$ be a category. We say that $\C$ is {\bf semi-monoidal} when it satisfies all the properties for a monoidal category except for the requirement of a unit object --- i.e. there exists a {\bf tensor}
$( \ \Box \ ) : \C\times \C \rightarrow \C$
together with a natural object-indexed family of {\bf associativity isomorphisms}
$\{ \tau_{A,B,C} : A\Box (B\Box C)\rightarrow (A\Box B) \Box C \} _{A,B,C\in Ob(\C)}$
satisfying MacLane's pentagon condition
\[  (\tau_{A,B,C}\Box 1_D) \tau_{A,B\Box C,D}  (1_A \Box \tau_{B,C,D}) \ = \  \tau_{A\Box B,C,D} \tau_{A,B,C\Box D} \] 

When there also exists a natural object-indexed natural family of {\bf symmetry isomorphisms} $\{ \sigma_{X,Y}:X\Box Y\rightarrow Y\Box X \}_{X,Y\in Ob(\C) }$
satisfying MacLane's hexagon condition
\[  
\tau_{A,B,C}\sigma_{A\Box B,C}\tau_{A,B,C} \ = \ (\sigma_{A,C}\Box1_B) \tau_{A,C,B} (1_A\Box \sigma_{B,C}) \]
we say that $(\C,\Box ,\tau,\sigma)$ is a {\bf symmetric semi-monoidal category}. 

A functor $\Gamma : \C \rightarrow \mathcal D$ between two semi-monoidal categories $(\C ,\Box_\C )$ and $(\mathcal D,\Box_\mathcal D)$ is called (strictly) {\bf semi-monoidal} when $\Gamma (f\Box_\mathcal C g) = \Gamma(f) \Box_\mathcal D \Gamma(g)$. All monoidal categories are semi-monoidal, but not vice versa; the relationship is precisely analogous to that between monoids and semigroups. When a semi-monoidal category does not contain a unit object, we call it {\bf unitless monoidal}. % There is an obvious forgetful functor that takes monoidal categories to unitless monoidal categories; one simply takes the full subcategory containing all objects except the unit.  Appendix A also demonstrates how every semi-monoidal category can be `completed' to give a monoidal category, via an adjoint to this forgetful functor, giving a general version of the specific construction  of Section \ref{JI} below. 

%A functor $\Gamma : \C \rightarrow \D$ between two semi-monoidal categories $(\C ,\otimes_\C )$ and $(\D,\otimes_\D)$ is called (strictly) {\bf semi-monoidal} when $\Gamma (f\otimes_C g) = \Gamma(f) \otimes_D \gamma(g)$.

When a semi-monoidal category has only one object, we call it {\bf untyped monoidal}, or simply {\bf untyped}. %When the unique object of an untyped category is the unit object for the tensor, we call it {\bf degenerate}. Clearly, all non-degenerate untyped monoidal categories are unitless.

%Finally,  a semi-monoidal category is called {\bf unsafely-typed} contains a (non-degenerate) unitless monoidal subcategory, and {\bf safely-typed} otherwise. 
\end{definition}

\begin{theorem} The structure $(\mathcal J,\_ \star\_ , \tau,\sigma)$, as given in Definition \ref{prelim} is an untyped symmetric monoidal category.
\end{theorem}
\begin{proof} This follows from Proposition \ref{coherent} above.
\end{proof}

\begin{remark} As observed in \cite{PH13}, we may construct similar structures based on congruence classes of the form $\{ p^k \mathbb N +x \}_{x=0\ldots p^k-1}$, for arbitrary $p\geq 2\in \mathbb N$, and  in general the untyped symmetric monoidal structures on the monoid of bijections on the natural numbers are in 1:1 correspondence with the interior points of the Cantor set (and thus are uncountably infinite).
We also refer to \cite{PH97,MVL} for many examples of these, given in terms of algebraic representations of inverse semigroups. As observed in the introduction, these are heavily used in models of  reversible computation and logic.% (the Linear Logic of \cite{LL,GOI1}). %similar single-object monoidal categories based on the representation theory of inverse semigroups.
\end{remark}

%The structure given in Definition \ref{prelim} is, as shown in Proposition \ref{coherent}, an untyped monoidal category, that we may denote $\mathcal J=(\mathcal J,\star,\tau,\sigma)$. %The motivation for this notation is explained in Remark \ref{isbell}.

%Semi-monoidal categories, with the notion of tensor-preserving functor given in Definition \ref{semicat} sbove clearly form a (large) category of which the category of monoidal categories is a subcategory.

%\begin{definition}
%Let $\Gamma : \C \rightarrow \mathcal D$ be a functor between two semi-monoidal categories $(\C ,\Box_\C )$ and $(\mathcal D,\Box_\mathcal D)$. We say that $\Gamma$ is (strictly) {\bf semi-monoidal} when $\Gamma (f\Box_\mathcal C g) = \Gamma(f) \Box_\mathcal D \Gamma(g)$. 
%Semi-monoidal categories, together with the notion of structure-preserving map given in Definition \ref{semicat}, form a (large) category that we denote $\bf SemiCat$. We denote the category of monoidal categories, together with strict monoidal functors by $\bf MonCat$. It is immediate that $\bf MonCat$ is a subcategory of $\bf SemiCat$.
%\end{definition}

\subsection{Coherence in unitless monoidal categories}
%\subsection{Does the absence of a unit destroy coherence?}
When working with semi-monoidal categories, it would be exceedingly useful to be able to rely on MacLane's coherence theorems, for both associativity and (when appropriate) symmetry.  A natural worry, therefore, is whether there is some exceedingly subtle interaction between the existence of a unit object, and the monoidal tensor, that means these theorems are not applicable in the absence of a unit object. 

Readers familiar with the proof of MacLane's coherence theorem for associativity will recall that associativity and the units conditions are treated individually, and so this is unlikely to be the case. A conclusive argument is provided by an Appendix to \cite{PHarxiv}, where the obvious procedure for adjoining a (strict) unit object to a semi-monoidal category is described, and proved to be adjoint to the equally obvious forgetful functor. Thus, a semi-monoidal category may be transformed into a monoidal category with no side-effects. 

Despite this, there is a subtlety about {\em untyped} monoidal categories that is worth observing. In \cite{MCL}, MacLane gives an argument, due to J. Isbell, for considering associativity up to canonical isomorphism, rather than up to strict identity. This argument was based on a denumerable object $D$ in the skeletal category of sets satisfying $D\otimes D = D$, and a proof that strict associativity at this object would force a collapse to a triviality (i.e. the unit object for this category). Isbell's argument was phrased in terms of a single category with categorical products --- an appendix to \cite{PH13} argues that this is the case in arbitrary untyped monoidal categories, and a full coherence result is given in \cite{PHarxiv}.

\subsection{Coherence in the untyped monoidal category $(\mathcal J,\star)$}%, and number-theoretic identities}
%Untyped monoidal categories illustrate some subtleties about coherence for associativity. 
 In Section \ref{prelim}, we have seen that canonical isomorphisms for the untyped monoidal category $(\J,\star)$ are simply arithmetic expressions, built using modular arithmetic. Thus, it is possible (albeit frequently tedious and complex -- see also Section \ref{complexity}) to verify whether or not a diagram commutes by direct calculation.  Fortunately, we are also able to use MacLane's coherence theorem for associativity to derive --- from basic categorical principles --- a large class of diagrams that are guaranteed to commute, and thus a large class of number-theoretic identities that are guaranteed to be true. 

However, we are not able to use the common simplification of the associativity theorem --- valid in a wide range of settings --- that states {\em all canonical diagrams commute}. Since all arrows of $\mathcal J$ have the same source and target, this would imply that all arrows built recursively from the set $\left\{\tau, (\_ \star \_),(\ )^{-1}\right\}$ are equal, and this is clearly not the case! Instead, we must use the full statement of MacLane's theorem, in order to give a large class of diagrams that are guaranteed to commute.

The coherence theorem for associativity is based on the free monogenic monoidal category.  As we are interested in the unitless case, we work with this category, with the unit removed. Readers unhappy with this are invited to adjoin a unit object to $\mathcal J$, apply the coherence theorem for associativity, and then remove the unit object.
\begin{definition}
We define $(\mathcal W,\Box)$, the {\bf free monogenic semi-monoidal category}, to be precisely MacLane's free monogenic monoidal category \cite{MCL}, with the unit object removed.  An explicit description follows:,
\begin{itemize}
\item {(\bf Objects)} These are non-empty binary trees over a single variable symbol $x$. Thus, $x\in Ob(\mathcal W)$, and, for all $a,b\in Ob(\mathcal W)$, the formal string $a\Box b\in Ob(\mathcal W)$.
\item {(\bf Arrows)} Given $w\in Ob(\mathcal W)$, the {\bf rank} of $w$ is the number of  occurrences of  the symbol $x$ within the string $w$, so $rank(x)=1$, and $rank(w)\geq 1$ for arbitrary $w\in Ob(\mathcal W)$.  There then exists a unique arrow between any two objects $a,b$ of the same rank, which we denote $(b\leftarrow  a)\in \mathcal W(a,b)$.
\item {\bf (Composition)} The composite of two unique arrows is simply the unique arrow with the appropriate source / target. Thus, $(c\leftarrow b)(b\leftarrow a)=(c\leftarrow a)$. 
\item {\bf (Tensor)} On objects, the tensor of $a$ and $b$ is the formal string $a\Box b$. The definition on arrows must then be $(b,a)\Box (v,u)=(b\Box v,a\Box u)$.
\item {\bf (Associativity isomorphisms)} The canonical isomorphism from $a\Box(b\Box c)$ to $(a\Box b)\Box c$ is the unique arrow between these two objects. 
\end{itemize}
The arrows between objects of rank $n$ correspond to the rebracketings of binary trees with $n$ leaves, in the obvious way. 
\end{definition}

There is then a natural semi-monoidal functor, $Sub:(\mathcal W,\Box)\rightarrow (\mathcal J,\star)$, the (unitless version of the) {\em Substitution functor} of \cite{MCL} p. 162.  Expanding out the abstract definition gives the following characterisation of this functor: 
\begin{itemize}
\item $Sub(w)=\mathbb N$, for all $w\in Ob(\mathcal W)$.
\item $Sub(w\leftarrow w) = id_\mathbb N $
\item $Sub(a\Box v \leftarrow a\Box u)=id_\mathbb N  \star Sub(v\leftarrow u)$
\item $Sub(v\Box a \leftarrow u\Box a)=Sub(v\leftarrow u)\star id_\mathbb N$
\item $Sub(a\Box b)\Box c \leftarrow a\Box(b\Box c)=\tau$.
\end{itemize}

MacLane's theorem states that $Sub(\mathcal W,\Box)\rightarrow (\mathcal J,\star)$ is indeed a (semi-) monoidal functor, and thus any diagram over $(\mathcal J,\star)$ that is the image of a diagram over $(\mathcal W,\Box)$ under this functor is guaranteed to commute. This simple result gives a countably infinite set of diagrams that are guaranteed to commute (and thus a corresponding set of arithmetic identities that are guaranteed to hold). For example, in $(\mathcal J, \star ,\tau)$, the following diagram commutes:
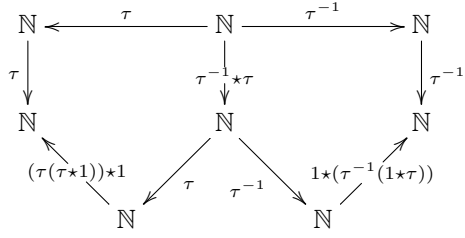
\[ \xymatrix{
\mathbb N 	\ar[d]_\tau		&								&	\mathbb N	\ar[ll]_\tau \ar[rr]^{\tau^{-1}} \ar[d]|{\tau^{-1}\star \tau}	&										& \mathbb N  \ar[d]^{\tau^{-1}}	\\
\mathbb N 				&								&	\mathbb N	\ar[dl]^\tau \ar[dr]_{\tau^{-1}}					&										& \mathbb N  		\\
					&\mathbb N \ar[ul]|{(\tau(\tau\star 1))\star 1}	&													&	\mathbb N	\ar[ur]|{1\star(\tau^{-1}(1\star \tau))}	&		  		\\
}
\]
To prove that this commutes, simply note that it is the image of the following diagram over $(\mathcal W,\Box)$

\noindent
\scalebox{0.85}
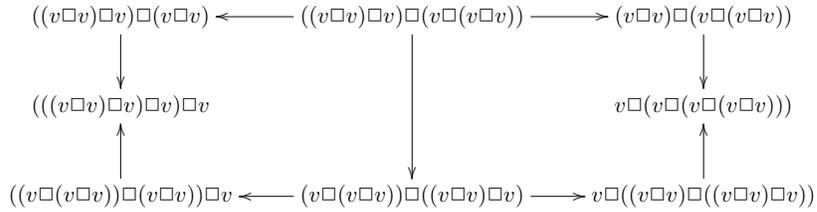
{
\xymatrix{
								& ((v\Box v)\Box v)\Box (v\Box v)  	\ar[d]	&	((v\Box v)\Box v)\Box (v\Box (v\Box v))	\ar[l] \ar[r] \ar[dd]		& (v\Box v)\Box (v\Box (v\Box v))  \ar[d]					& 	\\
								& (((v\Box v)\Box v)\Box v)\Box v			&													& v\Box (v\Box (v\Box (v\Box v)))						&  		\\
								& ((v\Box (v\Box v))\Box (v\Box v))\Box v \ar[u]	&		(v\Box (v\Box v))\Box ((v\Box v)\Box v)	\ar[l] \ar[r]		& v\Box ((v\Box v)\Box ((v\Box v)\Box v))\ar[u]				&		  		\\
}
}

(We do not label the arrows of this diagram, since they are uniquely determined by their source and target.  They may, of course, simply be thought of as re-bracketings of binary trees of rank 6).

\section{Number-theoretic identities via coherence}\label{complexity}

We have shown that MacLane's coherence theorem provides a countably infinite set of categorical diagrams that may be guaranteed to commute; however, the basic building blocks of these diagrams are the modular arithmetic operations of Definition \ref{prelim} --- thus the coherence theorem predicts identities within modular arithmetic. 
%\begin{remark}
It is of course possible to verify that such diagrams, such as the above diagram,  commute, using modular arithmetic and a case-by-case analysis,  as in Section \ref{prelim}. However, to prove the following identities
\[ \tau^2=((\tau(\tau\star 1))\star 1)\tau (\tau^{-1}\star \tau) \ \mbox{ and } \tau^{-2} = (1\star (\tau^{-1}(1\star \tau^{-1})))\tau^{-1}(\tau^{-1}\star \tau) \]
as expressed by this diagram, would involve working with a case-by-case analysis of modulo classes of the form $\{ n \ (mod\ 32) = k\}_{k=0\ldots 31}$. %and in general a diagram with $N$ nodes may involve working with  modulo classes of the form $\{ n \ (mod\ 2^N) = k\}_{k=0\ldots 2^N-1}$.  
The unfortunate referee assigned the task of verifying the calculations of Proposition \ref{coherent} will agree that this is a task to be avoided, if at all possible.

In general, a canonical diagram  with $N$ nodes may be the image of a diagram in $(\mathcal W,\Box)$ containing trees of depth $N$. An arithmetic check of the validity of this diagram may therefore require a case-by-case analysis that includes modulo classes  $\{ \mathbb N +x \ (mod\ K) \}_{x=0\ldots K}$, where $0\leq K < 2^N$. Clearly this is unfeasible, even for moderately large $N$.  However, when a diagram is indeed the image of a diagram in $(\mathcal W,\Box)$ %(and thus can be expressed in terms of a re-bracketings of binary trees), 
the coherence theorem for associativity allows us to assert equality between all paths within the diagram that have the same source and target --- and thus the correctness of the (somewhat complicated) corresponding arithmetic identities. 

Checking that an arbitrary diagram is within the image of this functor (and thus commutes) may be seen intuitively to be a much simpler task. In Section \ref{nextPapers} below, we suggest that this task is in fact {\em linear}, instead of {\em exponential}.
%\end{remark}

\begin{remark}
As well as the modular arithmetic identities predicted by the coherence theorem for associativity, it may be observed that $(\mathcal J,\star)$ is a {\em symmetric} untyped monoidal category, and thus the theory of coherence of symmetry will predict an additional countably infinite set of identities. This is indeed correct, and coherence for other categorical properties (e.g. the distributivity of $\times$ over $\uplus$) also provide further sets of arithmetic identities. The study of these is work in progress. %  forms the subject of a subsequent paper. %correct; however, the constructions of this  paper also serve partly as a motivational example for a general theory of coherence for typed and untyped monoidal structures;  we therefore treat associativity and symmetry separately, and study this and the required coherence issues in a subsequent paper.
\end{remark}

\section{Conclusions and future work}\label{nextPapers}
We have demonstrated that, working within a simple representative arithmetic example,  MacLane's coherence theorem predicts the correctness of a countably infinite set of identities in modular arithmetic. As observed in the introduction, this particular category is simply the simplest possible example of an uncountably infinite set of similar untyped monoidal categories based on reversible arithmetic functions of the natural numbers. Thus, there appears to be considerable scope for deriving arithmetic and number-theoretic  identities from categorical first principles.

Of equal interest -- both in the category we give, or in any similar category --- is whether  we can go in the opposite direction; given a canonical diagram expressing some identities of modular arithmetic, % (based on some fixed set of operations and rules for combining them),
 is there a partial or complete decision procedure that will tell us whether it is the image of some diagram under MacLane's substitution functor (and thus whether the arithmetic identities expressed are correct)?  We conjecture that not only is this the case, but that the complexity of this  decision procedure is linear in the number of edges of the diagram (this conjecture is based on an algorithm presented by the author at the conference \cite{DAG}, based on Robinson's unification algorithm \cite{JG}). 
 
 We also expect to find further applications in a number of other fields. In particular, %it is demonstrated in \cite{KB} that Thompson's group $F$ has an untyped monoidal tensor (although not phrased in these terms), and 
 constructions similar to those of this paper were used in an algebraic setting to give full concrete representations of Thompson's $V$ and $F$ groups \cite{verus1,verus2}. Thus, any results or decision procedures  for the abstract categorical theory can reasonably be expected to find applications to the theory of these groups.

 \bibliographystyle{plain}
\bibliography{modular_refs}

\begin{thebibliography}{10}

\bibitem{SA96}
S.~Abramsky.
\newblock Retracing some paths in process algebra.
\newblock In {\em CONCUR 96}, pages 1--17. Springer-Verlag Lecture Notes in
  Computer Science, 1996.

\bibitem{AHS}
S.~Abramsky, E.~Haghverdi, and P.~Scott.
\newblock Geometry of interaction and linear combinatory algebras.
\newblock {\em Mathematical Structures in Computer Science}, 12 (5), 2002.

\bibitem{DAG}
Samson Abramsky, Jean Krivine, and Michael~W. Mislove.
\newblock {Information Flow and Its Applications (Dagstuhl Seminar 12352)}.
\newblock {\em Dagstuhl Reports}, 2(8):99--112, 2013.

\bibitem{DR}
V.~Danos and L.~Regnier.
\newblock Local and asynchronous beta reduction.
\newblock In {\em Proceedings of the Eighth Annual IEEE Symp. on Logic in
  Computer Science}, 1993.

\bibitem{JG}
J.~Gallier.
\newblock {\em Logic for Computer Science}.
\newblock J. Wiley \& sons, 1987.

\bibitem{LL}
J.-Y. Girard.
\newblock Linear logic.
\newblock {\em Theoretical Computer Science}, 50:1--102, 1987.

\bibitem{GOI1}
J.-Y. Girard.
\newblock Geometry of interaction 1.
\newblock In {\em Proceedings Logic Colloquium Õ88}, pages 221--260.
  North-Holland, 1989.

\bibitem{GOI2}
J.-Y. Girard.
\newblock Geometry of interaction ii.
\newblock In {\em Proceedings COLOG 88}, Springer LNCS, pages 76--93.
  Martin-Lof \& Mints, 1989.

\bibitem{GOI0}
J.-Y. Girard.
\newblock Toward a geometry of interaction.
\newblock {\em Contemporary Mathematics}, 92:69--108, 1989.

\bibitem{PH97}
P.~Hines.
\newblock {\em The algebra of self-similarity and its applications}.
\newblock PhD thesis, University of Wales, Bangor, 1997.

\bibitem{PH99}
P.~Hines.
\newblock The categorical theory of self-similarity.
\newblock {\em Theory and Applications of Categories}, 6:33--46, 1999.

\bibitem{PH03}
P.~Hines.
\newblock A categorical framework for finite state machines.
\newblock {\em Mathematical Structures in Computer Science}, 13:451--480, 2003.

\bibitem{PHarxiv}
P.~Hines.
\newblock Coherence in hilbert's hotel.
\newblock {\em arXiv:1304.5954 [math.CT]}, 2013.

\bibitem{PH13}
P.~Hines.
\newblock Types and forgetfulness in categorical linguistics and quantum
  mechanics.
\newblock In M.~Sadrzadeh C.~Heunen, editor, {\em Categorical Information flow
  in Physics and Linguistics}, pages 215--248. Oxford University Press, 2013.

\bibitem{JSV}
A.~Joyal, R.~Street, and D.~Verity.
\newblock Traced monoidal categories.
\newblock {\em Mathematical Proceedings of the Cambridge Philosophical
  Society}, pages 425--446, 1996.

\bibitem{KL}
M.~Kelly and M.~Laplaza.
\newblock Coherence for compact closed categories.
\newblock {\em Journal of Pure and Applied Algebra}, 19:193--213, 1980.

\bibitem{LS}
J.~Lambek and P.~Scott.
\newblock {\em Introduction to Higher Order Categorical Logic}.
\newblock Cambridge University Press, 1986.

\bibitem{MVL}
M.~V. Lawson.
\newblock {\em Inverse semigroups: the theory of partial symmetries}.
\newblock World Scientific, Singapore, 1998.

\bibitem{verus1}
M.~V. Lawson.
\newblock Representations of the thompson group f via representations of the
  polycyclic monoid on two generators.
\newblock 2004.

\bibitem{verus2}
M.~V. Lawson.
\newblock Orthogonal completions of the polycyclic monoids.
\newblock {\em Communications in Algebra}, 35 (5), 2007.

\bibitem{MSPU}
P.~Urzyczyn M.~S\o\~rensen.
\newblock {\em Lectures on the Curry-Howard Isomorphism}, volume 149 of {\em
  Studies in Logic and the Foundations of Mathematics}.
\newblock Elsevier Science, 1998.

\bibitem{MCL}
S.~MacLane.
\newblock {\em {Categories for the working mathematician}}.
\newblock Springer-Verlag, New York, second edition, 1998.

\bibitem{FP}
F.~Paoli.
\newblock {\em Substructural Logics: A Primer}.
\newblock Kluwer, 2002.

\end{thebibliography}

\end{document}